\documentclass[a4paper,12pt]{article}
\usepackage{lineno}
\usepackage{hyperref}
\usepackage{mathrsfs}
\usepackage{amssymb}
\usepackage{amsmath}
\usepackage{amsfonts}
\usepackage{amsthm}
\usepackage{tikz-cd}
\usepackage{enumerate}

\theoremstyle{definition}
\newtheorem{defn}{Definition}[section]

\newtheorem{theorem}[defn]{Theorem}
\newtheorem{proposition}[defn]{Proposition}

\numberwithin{equation}{section}

\begin{document}

\title{A general Mayer-Vietoris sequence in algebraic $K$-theory}  

\author{
\textbf{Yakun Zhang}\\
\small \emph{School of Mathematics, Nanjing Audit University,}\\
\small \emph{Nanjing 211815, China}\\
\small \emph{E-mail: zhangyakun@nau.edu.cn}\\
}
\date{}
\maketitle

\noindent \textbf{Abstract.} 
This paper investigates the Mayer-Vietoris sequence for the Milnor square. While such sequences often involve elusive intermediate terms, we provide an explicit characterization of the key group $X$ in a new, more general variant of the sequence. By identifying $X$ as a categorical pullback, we provide a full, constructive proof of the modified Mayer-Vietoris sequence. Furthermore, we show that $X$ fits into a structural exact sequence involving the relative $K$-groups $K_{*}(A, B, I)$. Finally, we provide a homotopy-theoretic description of $X$ as the homotopy group of a suitable fiber, clarifying its structure, kernel , and image.



\section{Introduction}
Consider a Milnor square of associative rings
\begin{equation}\label{milnor}
\begin{tikzcd} 
A \arrow[r, "f"] \arrow[d] &  B \arrow[d]  \\
A / I \arrow[r, "\bar{f}"] &  B / I
\end{tikzcd}	
\end{equation}
where $I$ is a two-sided ideal of $A$ that is mapped isomorphically onto a two-sided ideal of $B$ under the ring homomorphism $f$, and the vertical maps are the canonical projections. This pullback square induces a commutative diagram of algebraic $K$-groups:
$$
\begin{tikzcd}
    K_{i+1}(A) \arrow[r] \arrow[d] & K_{i+1}(A/I) \arrow[r] \arrow[d] & K_i(A, I) \arrow[r, "\alpha_{i}"] \arrow[d, "\varepsilon_i"] & K_i(A) \arrow[r] \arrow[d] & K_i(A/I) \arrow[d] \\
    K_{i+1}(B) \arrow[r] & K_{i+1}(B/I) \arrow[r, "\beta_{i}"] & K_i(B, I) \arrow[r] & K_i(B) \arrow[r] & K_i(B/I)
\end{tikzcd}
$$
Here, $\varepsilon_i: K_{i}(A, I)\rightarrow K_{i}(B, I)$ denotes the $i$-th homomorphism of relative $K$-groups, $\mathrm{ker}(\varepsilon_i)$ is called the $i$-th excision kernel. Since $\varepsilon_i$ is not always an isomorphism, a Mayer-Vietoris-type long exact sequence cannot be derived directly from this diagram. To address this, Weibel \cite[Proposition 5.1]{1983K1} introduced a modified Mayer-Vietoris sequence for all $i\geq 0$:
\begin{equation} \label{mv1}
\cdots \to K_{i+1}(A/I) \oplus K_{i+1}(B) \rightarrow \mathrm{sub}\text{-}K_{i+1}(B/I) \rightarrow \mathrm{quo}\text{-}K_{i}(A) \rightarrow K_{i}(A/I) \oplus K_{i}(B) \to \cdots	
\end{equation}
where $\mathrm{sub}\text{-}K_{i+1}(B/I)$ denotes the subgroup of $K_{i+1}(B/I)$ consisting of elements mapping to $0$ under $K_{i+1}(B/I) \rightarrow K_{i-1}(A, B, I)$, and $\mathrm{quo}\text{-}K_{i}(A)$ is the quotient of $K_i(A)$ by the image of $K_{i}(A, B, I)$. However, he did not provide a complete and rigorous proof, indicating the argument only via commutative diagrams and diagram chasing.

In this paper, we re-examine these two groups from a categorical perspective. Specifically, $\mathrm{quo}\text{-}K_{i}(A) = K_{i}(A)/\alpha_{i}\bigl(\ker(\varepsilon_i)\bigr)$ can be characterized as the pushout:
\begin{equation}\label{pushout}
\begin{tikzcd}
 K_{i}(A, I) \arrow[r, "\alpha_{i}"] \arrow[d, two heads, "\pi_1"] &  K_{i}(A) \arrow[d, "\pi_2"]  \\
K_{i}(A, I)/\mathrm{ker}(\varepsilon_i) \arrow[r] &  \mathrm{quo}\text{-}K_{i}(A)   
\end{tikzcd}	
\end{equation}
Correspondingly, $\mathrm{sub}\text{-}K_{i+1}(B/I)  = \{x\in K_{i+1}(B/I)\mid \beta_{i}(x)\in \mathrm{Im } \varepsilon_i\}$ is the pullback:

\begin{equation}\label{pullback}
\begin{tikzcd}
   \mathrm{sub}\text{-}K_{i+1}(B/I) \arrow[r] \arrow[d, hook, "i_1"] & K_{i}(A, I)/\mathrm{ker}(\varepsilon_i)   \arrow[d, hook, "i_2"] \\
K_{i+1}(B/I)  \arrow[r, "\beta_{i}"] & K_{i}(B, I) 
\end{tikzcd}	
\end{equation}

The main contributions of this paper are as follows. First, we re-examine Weibel's result and provide a full, constructive proof of his variant of the Mayer-Vietoris sequence. Building upon this foundation, we further refine the construction to obtain a new, more general variant of the sequence. Finally, we employ homotopy-theoretic methods to provide a precise characterization of the mysterious third term in our new sequence. By identifying it with the homotopy group of a suitable homotopy fiber, we derive a new exact sequence that thoroughly clarifies its structure, kernel, and image.

\section{Generalized Mayer-Vietoris sequence}

Throughout this section, we maintain the notation established in the introduction regarding the Milnor square \eqref{milnor}. Specifically, the groups $\mathrm{quo}\text{-}K_{i}(A)$ and $\mathrm{sub}\text{-}K_{i+1}(B/I)$ are understood as the pushout \eqref{pushout} and pullback \eqref{pullback} constructions, respectively.

To prove our main results, we present two key propositions that underlie the following constructions. Their proofs rely only on standard techniques in homological algebra involving exact sequences, pullbacks, and pushouts. Since the arguments are routine and straightforward, we omit the detailed proofs and leave them to the reader. These propositions will be used to provide a rigorous diagram-theoretic proof for Weibel's Mayer-Vietoris sequence \eqref{mv1}. Building upon this foundation, we present our main result: a more general variant of the Mayer-Vietoris sequence that clarifies the internal structure of these terms.

In the following, a sequence of $R$-modules $A \to B \to C \to D$ is said to be exact if it is exact at $B$ and $C$, i.e., the image of each morphism is equal to the kernel of the succeeding one.

\begin{proposition}[\textbf{Stability of Exactness}] \label{exact1}
Let $R$ be a commutative ring and $A \xrightarrow{f} B \xrightarrow{g} C \xrightarrow{h} D$ be an exact sequence of $R$-modules.
\begin{enumerate}[(a)]
    \item Pullback Case: Let $i_2: C_1 \hookrightarrow C$ be a submodule inclusion. Let $B_1$ be the pullback of $g$ and $i_2$, with induced maps $i_1: B_1 \hookrightarrow B$ and $g_1: B_1 \to C_1$. Let $f_1: A \to B_1$ be the unique map with $i_1 \circ f_1 = f$, and $h_1 = h \circ i_2$. Then the following diagram commutes, and the top row is exact:
    $$
    \begin{tikzcd}
    A \arrow[r, "f_1"] \arrow[d, equal] & B_1 \arrow[r, "g_1"] \arrow[d, hook, "i_1"] & C_1 \arrow[r, "h_1"] \arrow[d, hook, "i_2"] & D \arrow[d, equal] \\
    A \arrow[r, "f"] & B \arrow[r, "g"] & C \arrow[r, "h"] & D
    \end{tikzcd}
    $$
    \item Pushout Case: Let $\pi_1: B \twoheadrightarrow B_2$ be a surjective homomorphism. Let $C_2$ be the pushout of $g$ and $\pi_1$, with induced maps $\pi_2: C \to C_2$ and $g_2: B_2 \to C_2$. Let $f_2 = \pi_1 \circ f$, and $h_2: C_2 \to D$ be the unique map with $h_2 \circ \pi_2 = h$. Then the following diagram commutes, and the bottom row is exact:
    $$
    \begin{tikzcd}
    A \arrow[r, "f"] \arrow[d, equal] & B \arrow[r, "g"] \arrow[d, two heads, "\pi_1"] & C \arrow[r, "h"] \arrow[d, "\pi_2"] & D \arrow[d, equal] \\
    A \arrow[r, "f_2"] & B_2 \arrow[r, "g_2"] & C_2 \arrow[r, "h_2"] & D
    \end{tikzcd}
    $$
\end{enumerate}
\end{proposition}

 The following proposition establishes the categorical dual results for lifting exactness through pullback and pushout constructions.

\begin{proposition}[\textbf{Lifting of Exactness}] \label{exact2}
Let $R$ be a commutative ring.
\begin{enumerate}[(a)]
    \item Pushout Case: Let $i_1: B_1 \hookrightarrow B$ be an injective homomorphism of $R$-modules. Let $C$ be the pushout of $g_1: B_1 \to C_1$ and $i_1$, with induced maps $i_2: C_1 \hookrightarrow C$ and $g: B \to C$. Let $f: A \to B$ be the unique map with $i_1 \circ f_1 = f$, and $h: C \to D$ be the unique map with $h \circ i_2 = h_1$. If the top row of the following commutative diagram is exact, then the bottom row $A \xrightarrow{f} B \xrightarrow{g} C \xrightarrow{h} D$ is exact:
    $$
    \begin{tikzcd}
    A \arrow[r, "f_1"] \arrow[d, equal] & B_1 \arrow[r, "g_1"] \arrow[d, hook, "i_1"] & C_1 \arrow[r, "h_1"] \arrow[d, hook, "i_2"] & D \arrow[d, equal] \\
    A \arrow[r, "f"] & B \arrow[r, "g"] & C \arrow[r, "h"] & D
    \end{tikzcd}
    $$
    
     \item Pullback Case: Let $\pi_2: C \twoheadrightarrow C_2$ be a surjective homomorphism of $R$-modules. Let $B$ be the pullback of $g_2: B_2 \to C_2$ and $\pi_2$, with induced maps $\pi_1: B \twoheadrightarrow B_2$ and $g: B \to C$. Let $f: A \to B$ be the unique map with $\pi_1 \circ f = f_2$, and $h = h_2 \circ \pi_2$. If the bottom row of the following commutative diagram is exact, then the top row $A \xrightarrow{f} B \xrightarrow{g} C \xrightarrow{h} D$ is exact:
    $$
    \begin{tikzcd}
    A \arrow[r, "f"] \arrow[d, equal] & B \arrow[r, "g"] \arrow[d, two heads, "\pi_1"] & C \arrow[r, "h"] \arrow[d, two heads, "\pi_2"] & D \arrow[d, equal] \\
    A \arrow[r, "f_2"] & B_2 \arrow[r, "g_2"] & C_2 \arrow[r, "h_2"] & D
    \end{tikzcd}
    $$
\end{enumerate}
\end{proposition}

\begin{theorem}
The Mayer-Vietoris sequence associated with the Milnor square \eqref{milnor} admits a more general variant of the form
\begin{equation} \label{mv2}
    \cdots \rightarrow K_{i+1}(A/I) \oplus K_{i+1}(B) \rightarrow X_i \rightarrow 
    K_{i}(A) \rightarrow K_{i}(A/I) \oplus K_{i}(B) \rightarrow \cdots,
\end{equation}
where $X_i$ is uniquely characterized as the pullback of $\bar{\partial}_i: \text{sub-}K_{i+1}(B/I) \rightarrow \text{quo-}K_{i}(A)$ and $\pi_i: K_{i}(A) \rightarrow \text{quo-}K_{i}(A)$. Furthermore, $X_i$ fits into an additional exact sequence
\begin{equation}\label{addition}
    \cdots \rightarrow K_{i}(A, B, I) \rightarrow X_i \xrightarrow{\varphi_i} K_{i+1}(B/I) \rightarrow K_{i-1}(A, B, I) \rightarrow \cdots, 
\end{equation}
with $\ker\,\varphi_i = \alpha_{i}\bigl(\ker(\varepsilon_i)\bigr)$ and $\text{im}\,\varphi_i = \text{sub-}K_{i+1}(B/I)$.
\end{theorem}

\begin{proof}
To establish the theorem, we proceed in three key steps: first, we rigorously verify the exactness of Weibel's variant \eqref{mv1}; second, we construct the generalized Mayer-Vietoris sequence \eqref{mv2} using the pullback characterization of $X_i$; and finally, we analyze the internal structure of $X_i$ by relating it to the homotopy fiber of a suitable map.

\textbf{Step 1: Exactness of Weibel's Variant \eqref{mv1}}

We begin by proving the exactness of Weibel's variant \eqref{mv1}, which serves as a foundational step for our subsequent arguments. As noted in the introduction, $\text{quo-}K_{i}(A)$ is understood as a pushout; combining this characterization with Proposition \ref{exact1}, we obtain a commutative diagram with exact rows, as shown below:
\[
\begin{tikzcd}
    K_{i+1}(A/I) \arrow[r] \arrow[d, equal] & K_{i}(A, I) \arrow[r, "\alpha_{i}"] \arrow[d, two heads, "\pi_1"] & K_{i}(A) \arrow[r] \arrow[d, "\pi_2"] & K_{i}(A/I) \arrow[d, equal] \\
    K_{i+1}(A/I) \arrow[r] & K_{i}(A, I)/\ker(\varepsilon_i) \arrow[r] & \text{quo-}K_{i}(A) \arrow[r] & K_{i}(A/I)
\end{tikzcd}
\]

Similarly, leveraging the pullback understanding of $\text{sub-}K_{i+1}(B/I)$ together with Proposition \ref{exact1}(a), we derive another commutative diagram with exact rows:
\[
\begin{tikzcd}
    K_{i+1}(B) \arrow[r] \arrow[d, equal] & \text{sub-}K_{i+1}(B/I) \arrow[r] \arrow[d, hook, "i_1"] & K_{i}(A, I)/\ker(\varepsilon_i) \arrow[r] \arrow[d, hook, "i_2"] & K_{i}(B) \arrow[d, equal] \\
    K_{i+1}(B) \arrow[r] & K_{i+1}(B/I) \arrow[r, "\beta_{i}"] & K_{i}(B, I) \arrow[r] & K_{i}(B)
\end{tikzcd}
\]

We then glue these two diagrams along their common term $K_{i}(A, I)/\ker(\varepsilon_i)$, resulting in an extended commutative diagram with exact rows (see below). A standard diagram chase on the middle two rows of this glued diagram confirms the exactness of Weibel's variant \eqref{mv1}, completing this foundational step of the proof.
\begin{center}
\begin{tikzcd}[row sep=1.2em, column sep=1.8em]
K_{i+1}(A) \arrow[r] \arrow[d, equal] & K_{i+1}(A/I) \arrow[r] \arrow[d, equal] & K_{i}(A, I) \arrow[r, "\alpha_{i}"] \arrow[d, two heads, "\pi_1"] & K_{i}(A) \arrow[r] \arrow[d, "\pi_2"] & K_{i}(A/I) \arrow[d, equal] \\
K_{i+1}(A) \arrow[r] \arrow[d, dashed, "f_{*}"] & K_{i+1}(A/I) \arrow[r] \arrow[d, dashed, "h_1"] & K_{i}(A,I)/\ker(\varepsilon_i) \arrow[r] \arrow[d, equal] & \text{quo-}K_{i}(A) \arrow[r] \arrow[d, dashed, "h_2"] & K_{i}(A/I) \arrow[d, dashed, "\bar{f}_{*}"] \\ 
K_{i+1}(B) \arrow[r] \arrow[d, equal] & \text{sub-}K_{i+1}(B/I) \arrow[r] \arrow[d, hook, "i_1"] & K_{i}(A,I)/\ker(\varepsilon_i) \arrow[r] \arrow[d, hook, "i_2"] & K_{i}(B) \arrow[r] \arrow[d, equal] & K_i(B/I) \arrow[d, equal] \\
K_{i+1}(B) \arrow[r] & K_{i+1}(B/I) \arrow[r, "\beta_{i}"] & K_{i}(B,I) \arrow[r] & K_{i}(B) \arrow[r] & K_i(B/I)
\end{tikzcd}
\end{center}

\vspace{0.5em}

\textbf{Step 2: Construction of the Generalized Mayer-Vietoris Sequence \eqref{mv2}}

Next, we define $X_i$ as the pullback of the homomorphisms $\bar{\partial}_i: \text{sub-}K_{i+1}(B/I) \rightarrow \text{quo-}K_{i}(A)$ and $\pi_i: K_{i}(A) \rightarrow \text{quo-}K_{i}(A)$, which uniquely characterizes $X_i$ as stated in the theorem. By Proposition \ref{exact2}(a), $X_i$ fits into the top row of the following commutative diagram:
\begin{center}
\begin{tikzcd}
    K_{i+1}(A/I) \oplus K_{i+1}(B) \arrow[r] \arrow[d] & X_i \arrow[r] \arrow[d] & K_{i}(A) \arrow[r] \arrow[d] & K_{i}(A/I) \oplus K_{i+1}(B) \arrow[d] \\
    K_{i+1}(A/I) \oplus K_{i+1}(B) \arrow[r] & \text{sub-}K_{i+1}(B/I) \arrow[r] & \text{quo-}K_{i}(A) \arrow[r] & K_{i}(A/I) \oplus K_{i}(B)
\end{tikzcd}
\end{center}

Since the bottom row of this diagram is exact (by construction and the results of Step 1), the exactness of the top row follows directly. This top row is precisely the generalized Mayer-Vietoris sequence \eqref{mv2} we aim to establish.

\textbf{Step 3: Internal Structure of $X_i$}

To characterize the internal structure of $X_i$ and derive the additional exact sequence involving $K_{i}(A, B, I)$, we follow the approach outlined in \cite[Theorem 2.1]{1988MV} — it is worth noting that this part essentially builds on the pullback conclusion of $X_i$ established in Step 2, but supplements new perspectives by introducing the homotopy fiber $U$ of the map $K(f\times \text{pr})$ to deepen the analysis. Recall that $K(A) = BQ\mathcal{P}(A)$ (with the standard modification for $i=0$) denotes the algebraic $K$-theory space, where $\pi_{i+1}(K(A)) = K_i(A)$ for all $i \ge 0$. We first consider the following commutative diagram, in which every row and column is a quasi-fibration sequence:
\begin{center}
\begin{tikzcd}[row sep=large, column sep=large]
K(A, B, I) \arrow[r] \arrow[d] & U \arrow[r] \arrow[d] & \Omega K(B / I) \arrow[d]\\
K(A, I) \arrow[r] \arrow[d] & K(A) \arrow[r] \arrow[d, "K(f\times \text{pr})"] & K(A / I) \arrow[d, "K(\bar{f}\times \text{id})"] \\
K(B, I) \arrow[r] & K(B \times A / I) \arrow[r, "K(\bar{f}\times \text{id})"] & K(B / I \times A / I)
\end{tikzcd}
\end{center}

The top quasi-fibration sequence induces a long exact sequence of homotopy groups:
\begin{equation*} 
    \cdots \rightarrow K_{i}(A, B, I) \rightarrow \pi_{i+1}(U) \rightarrow K_{i+1}(B/I) \rightarrow K_{i-1}(A, B, I) \rightarrow \cdots.
\end{equation*}

By definition of $\text{sub-}K_{i+1}(B/I)$, we have:
$$ \text{sub-}K_{i+1}(B/I) = \ker\left(K_{i+1}(B/I) \rightarrow K_{i-1}(A, B, I)\right) = \text{im}\left(\pi_{i+1}(U) \to K_{i+1}(B/I)\right).$$

To relate $\pi_{i+1}(U)$ to $X_i$, we construct the following commutative diagram:
\begin{equation*}
\begin{tikzcd}[row sep=large, column sep=large]
K_{i}(A, B, I) \arrow[r] \arrow[d] & \pi_{i+1}(U) \arrow[r, two heads] \arrow[d] & \text{sub-}K_{i+1}(B/I) \arrow[r, hook] \arrow[d, "\bar{\partial}"] & K_{i+1}(B, I) \arrow[d, "0"]\\
\alpha_{i}(\ker(\varepsilon_i)) \arrow[r, hook] & K_{i}(A) \arrow[r, two heads] & \text{quo-}K_{i}(A) \arrow[r] & K_{i}(A / I) 
\end{tikzcd}
\end{equation*}

Here, by the universal property of the pullback $X_i$, there exists a unique homomorphism $\pi_{i+1}(U) \rightarrow X_i$. Since both $\pi_{i+1}(U)$ and $X_i$ fit into the same exact sequence \eqref{mv2}, the Five Lemma implies that this homomorphism is an isomorphism, i.e., $\pi_{i+1}(U) \cong X_i$. This further confirms that the analysis of $X_i$'s internal structure is rooted in its pullback characterization, with the homotopy fiber serving as a supplementary tool.

Finally, the pullback property of $X_i$ induces an isomorphism between the kernels of the corresponding homomorphisms. Specifically, this yields $\ker\,\varphi_i \cong \alpha_{i}(\ker(\varepsilon_i))$, where $\varphi_i$ is defined in \eqref{addition}.  Combining this with the earlier result that $\text{im}\,\varphi_i = \text{sub-}K_{i+1}(B/I)$, we complete the proof of the theorem.
\end{proof}


\section*{Disclosure statement}
The author declares that there are no competing interests of a financial or personal nature.

\bibliography{mybibfile}
\end{document}